\documentclass[letterpaper,11pt]{amsart}

  \usepackage[margin=1.3in]{geometry}
  \usepackage{amsmath,amssymb,amsthm, mathrsfs, mathtools}
  \usepackage{amscd,mathtools}
  \usepackage{geometry}
  \usepackage{graphicx,epstopdf}
  \usepackage{color,xcolor}
  \usepackage{enumerate}
  \usepackage{xspace}
  \usepackage{tipa}
  \usepackage{epsfig}
  \usepackage{pinlabel}
  \usepackage[all]{xy}

  \usepackage{tikz}
  \usepackage{caption, subcaption}
  \usepackage{tikz-cd}
  \usetikzlibrary{calc,decorations.pathreplacing}
  
   \usetikzlibrary{shapes}
\usetikzlibrary{arrows}
\usetikzlibrary{calc}

  \usepackage{mathrsfs}
  \usetikzlibrary{matrix}
  \usetikzlibrary{positioning}
  \DeclareMathAlphabet{\pazocal}{OMS}{zplm}{m}{n}
  \tikzset{>=stealth}

  \usepackage{hyperref}

  \newcommand{\calA}{\mathcal{A}}
  \newcommand{\calB}{\mathcal{B}}

  \newcommand{\calF}{\mathcal{F}}

  \newcommand{\calS}{\mathcal{S}}

  \newcommand{\calX}{\mathcal{X}}

  \newcommand{\FF}{\mathbb{F}}

  \newcommand{\RR}{\mathbb{R}}

  \newcommand{\ZZ}{\mathbb{Z}}

  \newcommand{\gothic}{\mathfrak}

  \newcommand{\Gg}{{\gothic g}}

  \newtheorem{theorem}{Theorem}[section]
  \newtheorem{proposition}[theorem]{Proposition}
  
  \newtheorem{lemma}[theorem]{Lemma}

  \newtheorem{introthm}{Theorem}
  
  \newtheorem{introques}[introthm]{Question}

  \theoremstyle{definition}
  \newtheorem{definition}[theorem]{Definition}

  \newtheorem*{claim*}{Claim}
  
  \newtheorem{example}[theorem]{Example}
  
  \newtheorem*{question*}{Question}
  \newtheorem*{answer*}{Answer}
  \newtheorem*{application*}{Application}

  \theoremstyle{remark}
  \newtheorem{remark}[theorem]{Remark}
  \newtheorem*{remark*}{Remark}
  
  \newcommand{\secref}[1]{Section~\ref{#1}}
  \newcommand{\thmref}[1]{Theorem~\ref{#1}}
  
  \newcommand{\lemref}[1]{Lemma~\ref{#1}}
  \newcommand{\propref}[1]{Proposition~\ref{#1}}

  \newcommand{\exaref}[1]{Example~\ref{#1}}

  \newcommand{\figref}[1]{Figure~\ref{#1}}
  \newcommand{\defref}[1]{Definition~\ref{#1}}
  
  \newcommand{\eqnref}[1]{Equation~\eqref{#1}}
  \newcommand{\mf}[1]{\mathfrak{#1}}

  \DeclareMathOperator{\Mod}{Mod}

  \DeclareMathOperator{\twist}{twist}
  \DeclareMathOperator{\tw}{tw}
  \DeclareMathOperator{\fold}{fold}
  \DeclareMathOperator{\axis}{axis}  
  \DeclareMathOperator{\green}{stretch}

  \newcommand{\Aut}{\ensuremath{\operatorname{Aut}}\xspace} 
  \newcommand{\Inn}{\ensuremath{\operatorname{Inn}}\xspace} 

  \newcommand{\Teich}{{Teichm\"uller }} 
  \newcommand{\bd}{{\overline d}}

  \newcommand{\os}{{\overline s}}

\DeclareMathOperator{\Core}{Core}
\DeclareMathOperator{\Outer}{CV}
\DeclareMathOperator{\out}{Out}

\newcommand{\CV}{{\Outer_{\! n}}}
\newcommand{\Out}{\ensuremath{\out(\F)}}
    
  \newcommand{\param}{{\mathchoice{\mkern1mu\mbox{\raise2.2pt\hbox{$
  \centerdot$}}
  \mkern1mu}{\mkern1mu\mbox{\raise2.2pt\hbox{$\centerdot$}}\mkern1mu}{
  \mkern1.5mu\centerdot\mkern1.5mu}{\mkern1.5mu\centerdot\mkern1.5mu}}}

  \renewcommand{\setminus}{{\smallsetminus}}

  \newcommand{\from}{\colon\thinspace}

  \newcommand{\F}{{\FF_n}} 
  
  \newcommand{\BP}{{\mf{0}}}

\begin{document}


  \title[Quasi-geodesics in $\Out$ and their shadows in sub-factors]   
  {Quasi-geodesics in $\Out$ and their shadows in sub-factors}
  
%
 \author   {Yulan Qing}
 \address{Shanghai Center for Mathematical Sciences, Fudan University, Shanghai }
 \email{qingyulan@fudan.edu.cn}
 \author   {Kasra Rafi}
 \address{Department of Mathematics, University of Toronto, Toronto, ON }
 \email{rafi@math.toronto.edu}
%
 
  \date{\today}

\begin{abstract} 
We study the behaviour of quasi-geodesics in $\Out$. Given an element 
$\phi$ in $\Out$, there are several natural paths connecting the origin 
to $\phi$ in $\Out$; for example, paths given by Stallings' Folding Algorithm  
and paths induced by the shadow of greedy folding paths in Outer Space. We show that 
none of these paths is, in general, a quasi-geodesic in $\Out$. In fact, in contrast with 
the mapping class group setting, we construct examples where any quasi-geodesic 
in $\Out$ connecting $\phi$ to the origin will have to back-track in some free factor of $\F$. 
\end{abstract}
  
  \maketitle
  

\section{Introduction}
Let $\F = \langle s_1, s_2, ..., s_n \rangle$ denote the free group of rank $n$ and let $\Out$ denote 
the group of outer automorphisms of $\F$, 
\[
\Out : = \Aut(\F) / \Inn(\F).
\]
This group is finitely presented \cite{Nielsen}. For example, it can be generated by 
the set of right transvections and left
transvections:
\[
 \left \{
 \begin{array}{ll}s_i \rightarrow s_is_k \\
  s_j \rightarrow s_j \qquad \text{ for all } j \neq i \\
\end{array}
\right.
\hspace{1cm}  \left \{
 \begin{array}{ll}s_i \rightarrow s_k s_i \\
  s_j \rightarrow s_j \qquad \text{ for all } j \neq i \\
\end{array}
\right.
\]
together with elements permuting the basis and elements sending some $s_i$ to its inverse
$\overline{s}_i$. For an explicit presentation see \cite{McCool}.
Equip $\Out$ with the word metric associated to this generating set. We aim to 
understand the geometry $\Out$ as a metric space. Specifically, we want to understand 
the quasi-geodesics in $\Out$.

A common way to generate a path connecting a point in $\Out$ to the identity 
is to use the seminal idea of Stallings' Folding Algorithm \cite{stallingsfolding} which we discuss here briefly. One can create a model
for $\Out$ by considering the space of graphs $x$ of rank $n$ where the oriented edges
are labeled by elements of $\F$ inducing an isomorphism $\mu \from \pi_1(x) \to \F$
(defined up to conjugation and graph automorphism). We refer to $\mu$ as
the \emph{marking map}. For another such marked graph $x'$ with marking
map $\mu' \from \pi_1(x') \to \F$, we say $x'$ is obtained from $x$ by a \emph{fold} if 
there is a quotient map from $x$ to $x'$, 
\[\fold \from  x \to x',\] 
inducing an isomorphism $\fold^\star \from \pi_1(x) \to \pi_1(x')$ so that 
$\mu = \mu' \circ \fold^\star$.

Let $R_0$ be a rose with labels $s_1, \dots, s_n$ inducing an isomorphism 
\[
\mu_0 \from \pi_1(R_0) \to \F,
\] 
and, for $\phi \in \Out$, let $R = \phi(R_0)$ be the rose where the marking map 
\[
\mu \from \pi_1(R) \to \F
\qquad\text{is given by} \qquad \mu = \phi \, \mu_0.
\] 
Then, a sequence 
\[
R=x_m \to x_{m-1}  \ldots \to  x_1 \to x_0 = R_0
\] 
of folds produces a path in $\Out$ connecting the identity to $\phi$ as follows: 
for $0 <i \leq m$, consider a quotient map $x_i \to R_i$ by collapsing a spanning sub-tree, resulting in a rose $R_i$. Also, let $q_i \from x_i \to x_0=R_0$ be the composition of 
the folding maps. Then, we have the following diagram of homotopy 
equivalences:
\[
R_0 \to R_i \leftarrow x_i \to R_0
\]
where the first arrow is any graph automorphism between the two roses. 
We associate $x_i$ to the induced map $\phi_i \from \pi_1(R_0) \to \pi_1(R_0)$. 
The map $\phi_i \in \Out$
is coarsely well defined, depending on the chosen quotient maps 
$x_i \to R_i$ and the graph automorphism $R_0 \to R_i$. 
But the set of all possible resulting maps has a uniformly bounded diameter
in $\Out$ and the distance between $\phi_i$ and $\phi_{i+1}$ in $\Out$
is uniformly bounded. That is, we have a coarse path 
\[
\phi = \phi_m , \dots, \phi_0 ={\rm id}
\] 
in $\Out$ connecting $\phi$ to the identity which we refer to as a 
\emph{folding path} in $\Out$. 

Stallings Folding theorem provides an algorithm for finding a sequence of folds 
connecting any marked graph $x$ to $R_0$ where, in each fold, edges are identified 
according to their labels (see Section~\ref{Sec:Fold}). That is to say, the folds given by 
the Stallings algorithm are directed by the marking map. A path consisting of a sequence of marking-directed folds is called 
a \emph{Stallings folding path}. The ideas in Stallings' Folding Algorithm have found many interesting applications in the study of automorphisms of free groups. (see, for example, \cite{SF1} \cite{SF2}
\cite{SF3},\cite{SF4} and \cite{SF5}). 

Among the complexes frequently used to study $\Out$ (see for exmaple \cite{HM}, \cite{BFhyperbolicity} and \cite{KR}) are the free splitting complex and the free factor complex (see Section~\ref{FFandFS} for definitions). Let $\calS(\calA)$ denote the free splitting complex of $\F$. Stallings folding paths can be thought of as paths in $\calS(\calA)$. Handel and Mosher showed that these paths are  quasi-geodesics in $\calS(\calA)$ \cite{HM}. Let $\calF(\calA)$ denote the complex of free factors of $\F$. Kapovich and Rafi \cite{KR} proved that a projection from $\calS(\calA)$ to $\calF(\calA)$ takes geodesics to quasi-geodesics. Also, generalization of Stallings folding paths give rise to geodesics in $\CV$, for example see \cite{FMmetric} and \cite{convexball}.

However, Stallings folding paths do not in general give efficient paths in \Out. 

\begin{introthm} \label{Thm:SF} 
For any $K, C>0$, there exists an element $\phi \in \Out$ such that any Stallings folding 
path connecting $R=\phi(R_0)$ to $R_0$ does not yield a $(K, C)$--quasi-geodesic path
in \Out. 
\end{introthm}

\begin{proof}
Consider the following automorphism 
\[
\phi \colon 
\begin{cases}
a \to & a \\
b \to & b \\
c \to & c \, (ab^s)^t
\end{cases} 
\]
That is, if $R_0$ is the rose with edge labels $a$, $b$ and $c$, then
$R=\phi(R_0)$ has labels $a$, $b$, and $(ab^s)^t$. Following Stallings Folding Algorithm,  
to go from $R$ to $R_0$ we need to fold the third edge around the edges labeled $a$ 
and $b$. At each step, there is only one fold possible. The first and second edges remain
unchanged since they are also present in $R_0$. This takes $t\ (s+1)$ steps. That is, 
the associated path in $\Out$ has a length comparable to $t\, s$. 
  
However, one can see that there is another path connecting $R$ to $R_0$ that
takes $2s+t$ steps, namely:
\[
\big\langle a, b, c(ab^s)^t  \big\rangle \longrightarrow 
\big\langle ab^s, b, c(ab^s)^t \big\rangle \longrightarrow 
\big\langle ab^s, b, c \big\rangle \longrightarrow \langle a, b, c\rangle 
\]
Choosing $s, t$ sufficiently large compared with the given $K$ and $C$, we have shown 
that the path given by marking-directed folds was not a $(K,C)$--quasi-geodesic. 
\end{proof} 

One can also represent an element of $\Out$ using \emph{train-track maps} 
(see \cite{BH}) and consider a folding sequence according to the train-track structure. 
Or similarly, consider a geodesic in Culler-Vogtmann Outer Space and take
the shadow (that is, the difference of markings to $R_0$, see the discussion 
at the beginning of \secref{Sec:Outer-Space} and \figref{Fig:shadow})
of it to \Out. The same example above shows that
none of these paths would, in general, produce a quasi-geodesic in \Out. 

On the other hand, we observe in the above example, that even along the shorter paths, it takes at least 
$s$ steps to form or to eliminate an $s$-th power of $b$ and $t$ steps 
to eliminate a $t$-th power of $(ab^s)$. We examine this phenomenon through 
the language of \emph{relative twisting number} \cite{twistdefinition}. 

The \emph{relative twisting number} \cite{twistdefinition} of two labeled graphs  
$x$ and $x'$ around a loop $\alpha$ measures the difference between 
$x$ and $x'$ from the point of view of $\alpha$ and is denoted by $\tw_\alpha(x, x')$ 
(See \defref{Def:Twist}). We show that,  the length of any path in $\Out$ 
connecting $\phi$ to the identity is bounded below by the relative twisting number 
$\tw_\alpha(R, R_0)$ where $R= \phi(R_0)$. 

Again, instead of considering a path in the Cayley graph $\Out$, we consider a 
folding sequence $R_m, \dots, R_0$. This is a sequence of labeled graphs 
where $R_i$ is obtained from $R_{i+1}$ by a fold that it is not necessary coming
from Stallings' algorithm or any train-track structure (see  \secref{Sec:Fold}). 
We also show that if the length of the loop $\alpha$ remains long along the path, 
it takes longer to twist around $\alpha$. 

\begin{introthm}\label{Intro:slowtwist}
For any folding sequence $R_m, \dots, R_0$ and any loop $\alpha$, we have
\[
m \geq \tw_\alpha(R_0, R_n).
\] 
Further, if $\ell_{R_i}(\alpha)\geq L > 50$ for every $i$, then 
\[
m \geq \tw_\alpha(R_0, R_n)  \left( \log_{5} \frac{L}{50} \right). 
\] 
\end{introthm}

One might suspect that, in the above theorem, $\log_5 L$  can be replaced with $L$. 
However, we will show that the above inequality is sharp with an example 
(see \exaref{Exa:Log}). 

\thmref{Intro:slowtwist} can be viewed in the context of an attempt to have a distance
formula for the word length of an element in $\Out$ in analogy with the work of 
Masur-Minsky in the setting of the mapping class group \cite{hierarchy}. 
Let $S=S_{g,s}$ be a surface of genus $g$ with $s$ punctures and $\Mod(S)$ 
denote the mapping class group of $S$, that is, the group of orientation preserving 
self-homeomorphisms of $S$ up to isotopy. One can try to understand an element
$f \in \Mod(S)$ inductively by measuring the contribution of every subsurface to 
the complexity of $f$. This is done explicitly as follows: a marking $\mu_0$ on a surface 
$S$ is a set of simple closed curves that fill the surface, that is to say,
 every other curve on $S$ intersects some curve in $\mu_0$. 
 Masur and Minsky introduced a measure of complexity 
$d_Y(\mu_0, f(\mu_0))$ between $\mu_0$ and $f(\mu_0)$ called the subsurface 
projection distance. Namely, they defined a projection map 
\[
\pi_Y \from C(S) \to C(Y)
\]
from the curve graph of $S$ to the curve graph of a sub-surface $Y$ and defined 
$d_Y(\mu_0, f(\mu_0))$ to be the distance in $C(Y)$ between the projection 
$\pi_Y(\mu_0)$ and $\pi_Y(f(\mu_0))$ to $C(Y)$. 
They showed the sum of these subsurface projections is a good estimate for the 
word length of $f$ (see \cite{hierarchy} for more details).

To produce the upper-bound for the distance formula, Masur and Minsky constructed 
a class of quasi-geodesics called \emph{hierarchy paths}, whose lengths is the coarse 
sum of all subsurface projection distances. An important characteristic of 
these quasi-geodesics is that they do not back-track in any subsurface $Y$. 
That is, there is a quasi-geodesic, 
\[
[0,m] \to \Mod(S), \qquad i \to f_i,
\]
in $\Mod(S)$ so that the projection to the curve graph of $Y$
\[
[0,m] \to C(Y), \qquad i \to \pi_Y(f_i(\mu_0))
\]
is a quasi-geodesic for every subsurface $Y$ of $S$.

There are several analogues for the curve graph in the setting $\Out$, most
importantly, the free splitting graph $\calS(\F)$ and the free factor graph
$\calF(\F)$ both have been shown to be Gromov hyperbolic spaces \cite{HM, BFhyperbolicity}. 
For every sub-factor $\calA$ of $\F$, we have projection maps 
\cite{BFhyperbolicity}
\[
\Out \to \calS(\calA) \to \calF(\calA)
\] 
and it is known that every quasi-geodesic in $\calS(\F)$ projects to a quasi-geodesic 
in $\calF(\F)$ \cite{KR}. One may hope to construct quasi-geodesic in $\Out$
where the projections the free splitting or the free factor graph of a sub-factor is
always a quasi-geodesic. However, we use \thmref{Intro:slowtwist} to prove:

\begin{introthm}[Quasi-geodesics back-track in sub-factors]\label{Intro:backtrack}
For given constants $K_1$, $C_1$, $K_2$, $C_2 >0 $ there exists an automorphism, 
$\phi \in \Out$ such that, for any $(K_1, C_1)$--quasi-geodesic $p \from [0,m] \to \Out$ 
with $p(0) = id$ and $p(m) = \phi$, the shadow $\Theta_\calA \circ p$ of $p$ 
in $\calF(A)$ is not a $(K_2, C_2)$--reparameterized quasi-geodesic.
\end{introthm}

That is to say, there is $\phi \in \Out$  and a free factor $\calA$, such that every quasi-geodesic 
connecting $\phi$ to the identity backtracks in $\calF(\calA)$. In other words, 
there does not exists a quasi-geodesics between the identity and $\phi$ that 
projects to a quasi-geodesics in $\calF(\calA)$. Same is true for $\calS(\calA)$.

Another application of \thmref{Intro:slowtwist} is in the understanding of relationships 
between Outer Space geodesics and $\Out$ geodesics. The Outer Space, denoted 
$\CV$, is the set of metric graphs whose fundamental group is identified with $\F$.
It is a $CW$--complex with $\Out$ action and it was defined by Culler and Vogtmann
to study $\Out$ as an analogue of \Teich space which has $\Mod(S)$ action
\cite{outerspace}. One can project a path in $\CV$ to a path in $\Out$, by considering 
the associated difference of markings maps along a path in $\CV$. Bestvina and 
Feighn \cite{BFhyperbolicity} showed that greedy folding paths in $\CV$ projects to 
quasi-geodesics 
in free factor graphs for all sub-factors. This was used to produce a weak version
of a distance formula which give a lower bound for the word length in terms
of projection distance to $\calS(\calA)$ \cite{BBF}. 

\begin{figure}
\begin{tikzpicture}[
  node distance = 2.1cm,
  bubble/.style = {draw, ellipse, minimum width = 1.5cm, minimum height = 1.0cm},
  line/.style = {-latex'},
  ]
  \node(B0){$\Out$};
  \node[ right = of B0] (B1) {$\CV$};
  \node[ right = of B1] (B2) {$\calS(\calA)$};
  \node[below = of B2] (B3) {$\calF(\calA)$};
  \draw [line](B1) to [bend left=50] (B0.south) node[below, xshift=47, yshift=-22]{\tiny difference of markings to $R_0$};
  \draw [line](B0)--(B1.west) node[below, xshift=-30]{\tiny orbit of $R_{0}$};
  \draw [line] (B1) -- (B2.west)node[below, xshift=-27]{\tiny forget the metric};
  \draw [line] (B1) -- (B3.west)node[right, xshift=17, yshift=40]{\tiny projection map $\pi$};
\draw [line] (B2.south) -- (B3.north)node[below, xshift=-50, yshift=30, rotate=-47]{\tiny shortest primitive loop};
\end{tikzpicture}
\caption{A path in either $\Out$ or $\CV$ has an associated shadow path to 
other spaces via the maps described here.}
\label{Fig:shadow}
\end{figure}

However, it follows from \thmref{Intro:backtrack} that shadows of greedy folding paths
are not quasi-geodesics in $\Out$.

\begin{introthm} \label{Intro:CV}
The shadow of geodesics in $\CV$ do not in general behave well in $\Out$.
More specifically,
\begin{enumerate}[(i)] 
\item  For given constants $K$ and $C$, there are points $x, y \in \CV$ such that for every 
geodesic $[x, y]_{\CV}$ in $\CV$ connecting $x$ to $y$, its image in $\Out$ is not 
$(K, C)-$quasi-geodesic. 

\item There are points $x$ and $y$ in the thick part of $\CV$,  such that they are connected by a greedy
folding path whose shadow in $\Out$ is not a quasi-geodesic. 

\end{enumerate} 
\end{introthm}

Our methods do not say anything about the projection of quasi-geodesics in $\Out$ 
to $\calS(\F)$ or $\calF(\F)$. It is interesting to know if a geodesic in $\calS(\F)$ 
can be used as a guide to construct efficient paths in $\Out$. 

\begin{introques}
For a given $\phi \in \Out$, does there always exists a quasi-geodesic in 
$\Out$ connecting $\phi$ to the identity whose projection to $\calS(\F)$
is also a quasi-geodesic?
\end{introques} 

\subsection*{Acknowledgements}
The second author was partially supported by Discovery Grants from the 
Natural Sciences and Engineering Research Council of Canada (RGPIN 06486).  
We also thank the referee for many helpful comments. 

\section{Background} 
\subsection{Labeled graphs} \label{Sec:Label}
Recall from the introduction that a labeled graph $x$ induces an isomorphism 
$\mu \from \pi_1(x) \to \F$ called a \emph{marking}. Two labeled graphs $x$ and $x'$ are 
equivalent if there is a graph automorphism $f \from x \to x'$ such that 
the following diagram commutes up to free homotopy.
\[
\begin{tikzcd}
\pi_1(x'\,) \arrow{r}{\mu'} & \F \\
\pi_1(x) \arrow{ru}[below]{\ \mu} \arrow{u}{f^\star} &
\end{tikzcd}
\]
Let $w$ denote an element of $\F$. We refer to a conjugacy 
class $[w]$ of $w$ as a \emph{loop}. For any labeled graph $x$, 
and any loop $\alpha$, there is an immersion of a circle in $x$ representing $\alpha$
which (abusing the notation) we also denote by $\alpha$. We always assume 
this immersion to be the shortest in its free homotopy class in terms of the number 
of edges. The number of edges of a loop $\alpha$ in the given marked graph 
$x$ is denoted $\ell_x(\alpha)$, and is called the \emph{combinatorial length} of 
$\alpha$ in $x$. 

Sometimes it is more convenient to work with the universal cover of a marked graph. 
The universal cover of $x$ is an $\F$-tree (A simplicial tree with free $\F$ action).  
Given such a tree $T$, 
an element $w \in \F$ in the conjugacy class $\alpha$ acts hyperbolically on $T$, 
and we use $\axis_T(w)$ to denote the axis of its action. Consistent with the definition 
of combinatorial lengths in the graphs, we use $\ell_T(w)$ to denote the number of 
edges in a fundamental domain of the action of $w$. Note that this is independent 
of the choice of $w \in \alpha$. Hence, we can also use the notation $\ell_T(\alpha)$ 
which is equal to $\ell_x(\alpha)$. 

\subsection{Free Factor and Free Splitting Graphs}\label{FFandFS}
There are several analogues of the curve graph in the setting of $\Out$. 
The two important ones are the \emph{free factor graph} $\calF(\F)$ and
the free splitting graph $\calS(\F)$. 
 
A \emph{free factor} $\calA$ of $\F$ is a subgroup such that there exists another subgroup 
$\calB$ where
\[ \F = \calA \ast \calB.\]
A free factor $\calA$ is \emph{proper} in $\F$ if the rank of $\calA$ is strictly less than $n$. 
Then $\calF(\F)$ is a graph whose vertices are conjugacy classes of proper free factors of 
$\F$ and edges are associated to pairs of proper free factors where one is contained in 
the other. Similarly for each 
free factor $\calA$, one defines the free factor graph $\calF(\calA)$ whose vertices 
are conjugacy classes of proper free factors of $\calA$. 

A \emph{free splitting} over $\F$ is a minimal, simplicial (but possibly not free) 
action of the group $\F$ on a simplicial tree $T$ with trivial edge stabilizer. 
Then $\calS(\F)$ is a graph whose vertices are free splittings of $\F$ 
up to an equivariant isometry and two splittings are connected by an edge if one can 
be obtained from the other by a collapse map (see \cite{HM} for more details). As above, 
for any free factor $\calA$, $\calS(\calA)$ denotes the free splitting graph of $\calA$. 

There is a projection map 
\[
\pi \from \calS(\F) \to \calF(\F)
\]
defined as follows. Let $\alpha$ be a \emph{primitive} loop, that is, $\alpha = [w]$
and $\langle w \rangle$ is a free factor. Then, for any free splitting $T$, the translation
length of $\alpha$ in $T$, $\ell_T(\alpha)$, can be defined as before but
may be zero. We define $\pi(T) = \alpha$ where $\alpha$ is the primitive loop with 
shortest translation length. Similarly, there is a projection map 
$\calS(\calA) \to \calF(\calA)$ which we also denote by $\pi$. 
These maps are coarsely well defined (see \cite{BFhyperbolicity}) and further we have

\begin{theorem}[\cite{KR}]
The projection under $\pi$ of a quasi-geodesic in $\calS(\calA)$ is a re-parameterized
quasi-geodesic in $\calF(\calA)$. 
\end{theorem} 

\begin{remark}
The above theorem implies that, given a path $p$ in $\Out$, if its projection to the free 
factor graph is not a quasi-geodesic, then its projection to the free splitting graph was 
also not a quasi-geodesic.
\end{remark}

Let $x$ be a labeled graph and $T$ be the universal cover. For a proper free factor 
$\calA$ let $T|\calA$ be the minimal $\calA$--invariant subtree of $T$. Note that
$T|\calA \in \calS(\calA)$.  Also recall that the Outer Space 
(see \cite{outerspace} and the discussion in \secref{Sec:Outer-Space}) denoted 
$\CV$, is the set of volume 1 metric graphs whose fundamental group is identified with $\F$. Letting $T_0$ be the universal cover $R_0$, the labeled 
rose fixed in the introduction, we define a shadow map as follows:
\[
\Theta_\calA^\calS \from \Out \to \calS(\calA)
\qquad 
\phi \to  \phi(T_0)|\calA.
\]
That is, we change the action on $T_0$ according to $\phi$ and take the minimal 
$\calA$--invariant subtree. Composing with $\pi \from  \calS(\calA) \to \calF(\calA)$
we can define a shadow map to $\calF(\calA)$, 
\[
\Theta_\calA^\calF \from \Out \to \calF(\calA)
\qquad 
\phi \to  \pi \big( \Theta^\calS_\calA(\phi) \big).
\]
The map $\Theta_\calA^\calF$ is used more often and hence we shorten the notation
to $\Theta_\calA$. In fact, the shadow to $\calF(\calA)$ makes sense for all marked 
graphs. That is, we can define a map 
\[
\Theta^{*}_\calA \from \CV \to \calF(\calA)
\qquad 
\Theta^{*}_\calA(x) := \pi(T|\calA),
\]
where $T$ is the universal cover of $x$. Note that since the image  of $\Theta^{*}_\calA$ does not depend on the metric of the graph in question, it makes sense to write
$\Theta^{*}_\calA(x)$ even when $x$ is a marked graph with no metric. 
 
 We recall the lemma in \cite{BFhyperbolicity} on the upper bound for
the distance in the free factor graph. It states that if a loop $\alpha \in \calA$ is short
in $x$ then the shadow of $x$ in $\calF(\calA)$ is near $\alpha$. 

\begin{lemma} [Lemma 3.3 in \cite{BFhyperbolicity}] \label{Laddy}
Let $\calA$ be a proper free factor and $\alpha$ be a 
primitive element in $\calA$. Let $x$ be a marked graph so that 
$\ell_x(\alpha) \leq L$. Then 
\[ d_{\calF(\calA)} \big( \Theta^{*}_\calA(x), \alpha \big) \leq 6L + 13. \]
\end{lemma}

\subsection{Intersection Core}
To define the relative twisting, we first need to introduce the Guirardel 
Core associated to a pair of $\F$--trees. We give a characterization of the $2$--skeleton 
of the Guirardel Core that is different from (but equivalent to) the one given in 
\cite{Guirardelcore}. Given an $\F$--tree $T$, let $\mf{0}$ be a fixed vertex of $T$
that we call the base-point. We refer to the vertex $w(\BP)$ of $T$, $w \in \F$, 
simply by $w_T$ and we refer to an edge by a pair of words $(w, ws)_T$ where $s$ is 
the \emph{label} of the oriented edge. We say the edge $(w, ws)_T$ is an 
\emph{$s$-edge}. We say an $s$--edge $(w, ws)_T$ is \emph{preceded} by a 
$t$--edge $(u, ut)_T$ if $w=ut$. 

There is a one-to-one correspondence between the set of infinite geodesic rays originating 
from $\BP$ and the set of infinite freely reduced words in $\F$. Hence, the Gromov 
boundary of the group $\partial \F$, i.e. the equivalence class of quasi-geodesics 
rays in a Cayley graph $\F$, can be identified with the set of all geodesic rays 
starting from $\BP$ which in turn can be identified with the Gromov boundary 
$\partial T$ of $T$. An (oriented) edge $e= (w, ws)_T$ in $T$ defines a decomposition of 
$\partial \F$ into two sets in the following way: a vertex $v_T$ in $T$ is in \emph{front} of 
$e$ if the geodesic connecting $v_T$ and $(ws)_T$ does not contain the edge $e$; likewise 
a vertex $v_T$ is \emph{behind} $e$ if the geodesic connecting $v_T$ and $w_T$ does 
not contain $e$.  Let $\partial^+(e)$ be the set of all the geodesic rays originating from 
$\BP$ that eventually lie in front of $e$ and $\partial^-(e)$ be set of all geodesic rays 
originating from $\BP$ that eventually lie behind $e$. 

Note that the sets $\partial^{\pm}(e)$ are independent of the choice of $\BP$.  Also, 
\[
\partial^+(e) \cup \partial^-(e) = \partial \F
\qquad\text{and}\qquad 
\partial^+(e) \cap \partial^-(e) = \emptyset
\]
That is, $e$ induces a partition of $\partial \F$. 

\begin{definition}\label{Definition:Partition}
Consider an edge $e_1$ in $T_1$ and an edge $e_2$ in $T_2$. We say 
$e_1$ and $e_2$ are \emph{boundary equivalent} if they induce the same partition 
of $\partial \F$. We say $e_1 \times e_2$ is an \emph{intersection square} 
if all of the following four intersections, as subsets of $\partial \F$, are nonempty:
 \begin{align*}
 \partial^+(e_1) &\cap \partial^+(e_2) \neq \emptyset 
 &\partial^+(e_1) &\cap \partial^-(e_2)\neq \emptyset  \\
 \partial^-(e_1) &\cap   \partial^+(e_2)\neq \emptyset 
 &\partial^-(e_1) &\cap   \partial^-(e_2)\neq \emptyset 
 \end{align*}
\end{definition}
It follows from definition that a pair of boundary equivalent edges never form an intersection square.
\begin{example}
Consider the pair of trees $T_1$ and $T_2$ that are covers of roses 
$R_1$ and $R_2$ with edge labels $\langle  a, a^{2}b \rangle$ and 
$\langle a, b\rangle$ respectively. Let $\BP_1$ be the base point in $T_1$ and $\BP_2$ 
be the base points in $T_2$. 
 \begin{figure}[h!]
\begin{subfigure}{.4\textwidth}
\begin{tikzpicture}
 \tikzstyle{vertex} =[circle,draw,fill=black,thick, inner sep=0pt,minimum size=.5 mm]
 
[thick, 
    scale=1,
    vertex/.style={circle,draw,fill=black,thick,
                   inner sep=0pt,minimum size= .5 mm},
                  
      trans/.style={thick,->, shorten >=6pt,shorten <=6pt,>=stealth},
   ]
     \tikzstyle{vertexr} =[circle,draw,fill=red,thick, inner sep=0pt,minimum size=.9 mm]
 
[thick, 
    scale=1,
    vertex/.style={circle,draw,fill=red,thick,
                   inner sep=0pt,minimum size= .9 mm},
                  
      trans/.style={thick,->, shorten >=6pt,shorten <=6pt,>=stealth},
   ]

    \node[vertexr] (a) at (0, 0) {};
    \node[vertex] (b) at (1,0) {};
     \node[vertex] (c) at (2,0) {};
      \node[vertex] (d) at (0,1) {};
\draw [->, red](a)--(b);
\draw [->](b)--(c);
\draw [->](a)--(d);
\node at (-.08, -0.3){$\BP_1$};
\node at (2.8, -0.5){$T_1$};
\node at (0.5, 0.2){$a$};
\node at (1.5, 0.2){$a$};
\node at (0.3, 0.5){};
\node at (-0.4, 0.6){$a^2 b$};
\node at (0.5, -0.2){$e_{1}$};
\end{tikzpicture}
\end{subfigure}
\begin{subfigure}{.4\textwidth}
\begin{tikzpicture}
 \tikzstyle{vertex} =[circle,draw,fill=black,thick, inner sep=0pt,minimum size=.5 mm]
 
[thick, 
    scale=1,
    vertex/.style={circle,draw,fill=black,thick,
                   inner sep=0pt,minimum size= .5 mm},
                  
      trans/.style={thick,->, shorten >=6pt,shorten <=6pt,>=stealth},
   ]
   
   \tikzstyle{vertexr} =[circle,draw,fill=red,thick, inner sep=0pt,minimum size=.9 mm]
 
[thick, 
    scale=1,
    vertex/.style={circle,draw,fill=red,thick,
                   inner sep=0pt,minimum size= .9 mm},
                  
      trans/.style={thick,->, shorten >=6pt,shorten <=6pt,>=stealth},
   ]

    \node[vertexr] (a) at (0, 0) {};
    \node[vertex] (b) at (1,0) {};
     \node[vertex] (c) at (2,0) {};
      \node[vertex] (d) at (0,1) {};
\draw [->](a)--(b);
\draw [->, red](b)--(c);
\draw [->](a)--(d);
\node at (-.08, -0.3){$\BP_2$};
\node at (2.8, -0.5){$T_2$};
\node at (0.5, 0.2){$a$};
\node at (1.5, 0.2){$a$};
\node at (1.5, -0.2){$e_{2}$};
\node at (-0.2, 0.6){$b$};
\end{tikzpicture}
\end{subfigure}
\caption{The edges $e_1$ in $T_1$ and $e_2$ in $T_2$ form an intersection square.}
\end{figure}
 
We check that the edge $e_{1}=({\rm id}, a)_{T_1}$ and 
$e_{2}=(a, a^2)_{T_2}$ form an intersection square by producing the 
beginning of the infinite reduced words that lie in each of the intersections:
\begin{align*}
 \partial^+(e_1) &\cap \partial^+(e_2): a^{2} b^{-1} \dots 
 &\partial^+(e_1) &\cap \partial^-(e_2):a(a^{2}b)^{-1}aa \dots =ab^{-1} \dots\\
 \partial^-(e_1) &\cap   \partial^+(e_2):a^{2}b \dots
 &\partial^-(e_1) &\cap   \partial^-(e_2):b^{-1} \dots
 \end{align*}
Intuitively, this is because the edge $({\rm id}, a^2b)_{T_1}$ in $T_1$ that is behind 
$e_1$ maps to an edge path that passes $e_2$, but $e_1$ maps to an edge
that is behind $e_2$. Therefore, $e_1$ and $e_2$ are tangled to each other. 
\end{example}

Let $\overline{\Core(T_1, T_2)}$ be the sub-complex of $T_1 \times T_2$ that is the 
union of all intersection squares:
\[
\overline{\Core(T_1, T_2) }= 
\Big\{ e_1 \times e_2  \, \big|\,e_i \in T_i, \quad e_1 \times e_2 \text{  is an intersection square} \Big\}
\]
We define the \emph{intersection core}, or simply the \emph{core} of $T_1, T_2$,
to be:
\[
\Core(T_1, T_2) = \overline{\Core(T_1, T_2)} / \F
\]
It follows from Lemma 3.4 in \cite{uniformsurgery} that the above definition is the 
same as the definition given by Guirardel. 

It is clear from the definition that the core is symmetric: $\Core(T_1, T_2)$ is 
isomorphic to $\Core(T_2, T_1)$. 
For an edge $e_2 \in T_2$, the \emph{$e_2$--slice} of the 
intersection core is the subtree in  $T_1$ that is a collection of edges that form 
intersection squares with $e_2$:
\[
C_{e_2}(T_1) = 
  \Big\{ e \in T_1 \,\Big|\, e \times e_2 \text{ is a square in } 
  \overline{\Core(T_1, T_2)} \Big\}.
\]
Guirardel showed \cite{Guirardelcore} that the $e_2$--slice is always convex
and finite. That is, if two edges $e, e' \in T / \F$ 
are in a given slice, then all the edges on the geodesic path in $T_1/\F$ connecting 
$e$ and $e'$ are also in the slice.

\subsection{Stallings folding path}\label{Sec:Fold}
We describe Stallings folding path here as needed in this paper; for full generality, 
see ~\cite{stallingsfolding}. Let $x$ be a labeled directed graph. 
Then the edges of $x$ can be subdivided into \emph{edgelets} where each
edgelet is labeled with an element in the fixed basis $\{s_1, \dots, s_n\}$ of $\F$
and the concatenation of edgelets into an edge yields the original labeling of the edge. 
A \emph{marking-directed fold} from $x$ to $x'$ is a map from $x$ to $x'$ that identifies two 
edgelets  $e_1, e_2$ for which both of the following are satisfied:
\begin{enumerate}[(i)]
\item $e_1$ and $e_2$ share the same origin vertex.
\item $e_1$ and $e_2$ shares the same label
\end{enumerate}
The resulting quotient map from $x$ to $x'$ is a homotopy equivalence respecting the
markings.  We recall Stallings' Folding Theorem per the context of this 
paper(~\cite{stallingsfolding}).

\begin{theorem}[Stallings' Folding Theorem~\cite{stallingsfolding}] 
For any labeled graph $x$, there exists a finite sequence of marking-directed folds
$x = x_m \rightarrow \dots  \rightarrow x_1 \rightarrow x_0=R_0$ 
connecting $x$ to $R_0$.
\end{theorem}

For the labeled rose $R$ in \thmref{Thm:SF}, there is a unique Stallings folding path
connecting $R$ to $R_0$ because at every step, there are only two edgelets with
the same label and the same original vertex.

\subsection{Folding for roses}
Here, describe a folding map and its association to partitions of $\partial \F$
which is similar to marking-directed fold but one does not need to match the labels of the edges identified. 
For simplicity, we restrict our attention to cases where the labeled graph is a rose
which is sufficient for all our examples. 

\begin{definition}
Consider a labeled rose $R$ and choose two edges of $R$ labeled $s$ and $t$. 
Let $R'$ be a labeled rose whose edge labels are the same as $R$ except the 
edge-label $t$ has changed to $\overline{s}t$. We say $R'$ is obtained from $R$ by a 
\emph{fold}, and write 
\[R' =\fold(R, t, s) \quad \text{or simply}\quad R' = \fold(R).\]
We also write $T' = \fold(T, t, s)$ for the equivariant map in the universal covers.
\end{definition}

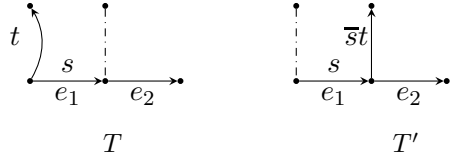
\begin{figure}[h!]
\begin{subfigure}{.2\textwidth}
\begin{tikzpicture}
 \tikzstyle{vertex} =[circle,draw,fill=black,thick, inner sep=0pt,minimum size=.5 mm]
 
[thick, 
    scale=1,
    vertex/.style={circle,draw,fill=black,thick,
                   inner sep=0pt,minimum size= .5 mm},
                  
      trans/.style={thick,->, shorten >=6pt,shorten <=6pt,>=stealth},
   ]
   
    \node[vertex] (a) at (0, 0) {};
    \node[vertex] (b) at (1,0) {};
     \node[vertex] (c) at (2,0) {};
      \node[vertex] (d) at (0,1) {};
       \node[vertex] (e) at (1,1) {};
\draw [->](a)--(b);
\draw [->](b)--(c);
\path[->](a) edge [bend right] (d);
\draw[dash dot](e)--(b);
\node at (0.5, 0.2){$s$};
\node at (0.5, -0.2){$e_1$};
\node at (1.5, -0.2){$e_2$};
\node at (-0.2, 0.6){$t$};
\end{tikzpicture}
\caption*{$T$}
\end{subfigure}
\qquad
\begin{subfigure}{.2\textwidth}
\begin{tikzpicture}
 \tikzstyle{vertex} =[circle,draw,fill=black,thick, inner sep=0pt,minimum size=.5 mm]
 
[thick, 
    scale=1,
    vertex/.style={circle,draw,fill=black,thick,
                   inner sep=0pt,minimum size= .5 mm},
                  
      trans/.style={thick,->, shorten >=6pt,shorten <=6pt,>=stealth},
   ]
   
    \node[vertex] (a) at (0, 0) {};
    \node[vertex] (b) at (1,0) {};
     \node[vertex] (c) at (2,0) {};
      \node[vertex] (d) at (0,1) {};
       \node[vertex] (e) at (1,1) {};
\draw [->](a)--(b);
\draw [->](b)--(c);
\draw[dash dot](d)--(a);
\draw[->](b)--(e);
\node at (0.5, 0.2){$s$};
\node at (0.5, -0.2){$e_1$};
\node at (1.5, -0.2){$e_2$};
\node at (0.8, 0.6){$\overline{s}t$};
\end{tikzpicture}
\caption*{$T'$}
\end{subfigure}
\end{figure}

\begin{remark}
In this paper, a fold can occur between any two edges. In contrast, 
marking-directed folds \cite{stallingsfolding} follow the labeling of the graphs.  
\end{remark}
  
Given $\F$--trees $T$, $ T'$, and a fixed base-points $\BP \in T$ and $\BP' \in T'$, 
there is a natural morphism $f \from T \to T'$ constructed as follows. 
Send $\BP$ to $\BP'$ and, for $ w\in \F$, send 
the vertex $w_T$ in $T$ to the vertex $w_{T'}$ in $T'$. Also, send an edge 
$(w, ws)_T$ to the unique embedded edge path connecting $w_{T'}$ to $(ws)_{T'}$ 
in $T'$. The morphism $f$ also  induces an $\F$--equivariant homeomorphism 
\[f_{\infty} \from \partial T \to \partial T'.\]

Let $T$ and $T'$ be the universal covers of $R$ and $R'$ respectively. 
A fold from $R$ to $R'$ induces a morphism from $T$ to $T'$ where an edge
of the form $(w, wt)_T$ is mapped to the edge path 
\[
\big[ (w, ws)_{T'}, (ws, ws(\overline{s}t))_{T'}\big] 
\]
and every other edge is mapped to a single edge. Similarly, the edge $(w, ws)_{T'} \in T'$
has two pre-images, $(w, ws)_T$ and a half of $(w, wt)_T$. All other edges have exactly one 
pre-images. We now describe how partitions given by edges in $T$ differ from that 
of edges in $T'$. 

\begin{proposition}\label{boundaryequivalent}
Let $R' =\fold(R, t, s)$ and $f \from T \to T'$ be the above morphism. Then:
\begin{enumerate}[(i)]
  \item If an edge $e$ in $T$ is not an $s$--edge or a $t$--edge then
   $f(e) \in T'$ is boundary equivalent to $e$. 
 \item If $e=(w, wt)_T$ is a $t$--edge, then $e' = (ws, ws(\overline{s}t))_{T'}$, 
 which is contained in $f(e)$, is boundary equivalent to $e$.
 \item If $e_2$ is an $s$--edge and $e_1$ is the $t$--edge starting at the same
 vertex as $e_2$ then $e'=f(e_2)$ partitions $\partial \F$ in the following way
 \[
 \partial^-(e') =  \partial^-(e_2) \setminus \partial^+(e_1)
 \qquad\text{and}\qquad
 \partial^+(e') =  \partial^+(e_2) \cup \partial^+(e_1).
 \]
 \item For adjacent edges $e_1 = (u, us)_T$, $e_2 = (us,w)_T$ and 
 $e' \subset f(e_2)$, we have 
 \[
 \partial^-(e_1) \subset \partial^-(e'). 
 \]
  \end{enumerate}
 \end{proposition}
\begin{proof}
Let $\BP$ and $\BP'$ be the base points in $T$ and $T'$; 
$F(\BP) = \BP'$. Note that changing a base-point $\BP$ to a point $w$ 
in $T$ moves both the partition given by $e$ in $T$ and given by $e'$ in $T'$ by 
the action of $w$. Hence, it is sufficient to prove the statement
for any desired base-point. 

Consider an edge $e' \in T'$ that has only one pre-image under 
the morphism $f \from T \to T'$. That is, $e$ is the only edge where
$e' \subseteq f(e)$. Then, $e$ and $e'$ are boundary equivalent. 
To see this, assume $\BP = w_T$.  A ray $r$ starting from $\BP$ crosses $e$ if an only 
if $f(r)$ crosses $e'$. Hence, a ray is eventually in front of $e$ if and only if it is 
eventually in front of $e'$ which, by definition, means $e$ and $e'$ are boundary 
equivalent. This finished the proof of the first two parts. 

Let $e_1=(u, ut)_T$ and $e_2=(u, us)_T$. We choose $\BP = u$.
Then $e_1$ and $e_2$ are the only two edges that are mapped over $f(e_1)$. 
But a ray $r$ in $T$ starting from $\BP$ crosses at most one of $e_1$ or $e_2$. 
In fact, $f(r)$ crosses either edge if and only if it is eventually in front of $f(e_1)$. 
That is 
\[
\partial^+{f(e)} =    \partial^+(e) \cup \partial^+(\hat e).
\]
The other equality in part (iii) holds because $\partial^-{f(e_1)}$ is the complement 
of $\partial^+{f(e_1)}$. 

In part (iv), we have $e_1=(u, us)_T$. Since $R' =\fold(R, t, s)$, the only ray $r$ in 
$\partial^-(e_1)$ where $f(r)$ crosses $f(e_1)$ is a ray starting with $(u, ut)_T$. 
If such is the case, then $f(r)$
starts with $\big[ (u, us)_{T'}, (us, us(\overline{s} t))_{T'}\big]$. But $e'$ is different from 
$(us, us(\overline{s} t))_{T'}$ (which has only one pre-image). Hence, $f(r)$ does not
cross $e'$ and lands in $\partial^+(e')$. Therefore,
 \begin{equation*}
 \partial^-(e_1) \subset \partial^-(e').  \qedhere
 \end{equation*}
\end{proof}

\section{Twisting estimate}
Let $|G|$ denote the number of edges in the given graph $G$. We now define 
\emph{relative twisting number}, which is an analogue of the Masur-Minsky twisting 
number \cite{hierarchy}. Recall that, for a given pair of trees $T$ and $T_0$,  
a slice $C_{e_0}(T)$ over an edge $e_0 \in T_0$ 
is the subtree of $T$ consisting of edges that form intersection squares with $e_0$. 

\begin{definition} \label{Def:Twist} 
Given a loop $\alpha$ and two $\F$--trees $T, T_0$, the 
\emph{relative twisting number} of $\F$--trees $T, T_0$ around $\alpha$ is
\[
\twist_{\alpha}(T, T_0) =
   \max_{e_0 \in T_0, \, w \in \alpha} \frac{|\axis_T(w) \cap C_{e_0}(T)|}{\ell_T (\alpha)}.
\]
If $T$ and $T_0$ are universal covers of $R$ and $R_0$, we define 
$\twist_{\alpha}(R, R_0) =\twist_{\alpha}(T, T_0)$. 
\end{definition}

\begin{example}\label{twistingexample}
We illustrate the intuition behind this definition by an example. 
Consider the rose $R$ with labels $\langle a, b, c\, (ab)^{4} \rangle$
and let $T$ be the universal cover of $R$ with basepoint $\BP$ and let $T_0$ be 
the universal cover of $R_0$ with labels $\langle a, b, c \rangle$ and base-point
$\BP_0$ as usual. Let $\alpha = [ab]$ and $w=ab$. Then, $\ell_T (\alpha) = 2$. 
Consider the edges $e_{a} = ({\rm id}, a)_{T_0}$. Then $C_{e_{a}}$ is the tree 
depicted below. The solid edges form a slice but the dashed edges are 
not in the core. One can see this by direct computation or from 
\cite[Lemma 3.7]{uniformsurgery} which states that $C_{e_a}$ is the interior 
of the convex hull of the pre-image of $e_a$. Intuitively, 
this is caused by the fact that 
the edge $( \overline c, (ab)^4)_{T}$ in $T$ (which is the doted edge on the right)
is in front of all the solid edges in $T$ but maps to an edge path that 
crosses $e_a$ which is behind the image of the solid edges in $T_0$. 
The edge $e=({\rm id}, a)_T$ (the dotted edge in 
the left) does not form an intersection square with $e_a$ because 
$\partial ^- e \cap \partial^+e_a = \emptyset$.

\begin{figure}[h!]
\begin{tikzpicture}[scale=0.8]
 \tikzstyle{vertex} =[circle,draw,fill=black,thick, inner sep=0pt,minimum size=.5 mm]
 
[thick, 
    scale=1,
    vertex/.style={circle,draw,fill=black,thick,
                   inner sep=0pt,minimum size= .5 mm},
                  
      trans/.style={thick,->, shorten >=6pt,shorten <=6pt,>=stealth},
   ]
   
    \tikzstyle{vertexr} =[circle,draw,fill=red,thick, inner sep=0pt,minimum size=.8 mm]
 
[thick, 
    scale=1,
    vertex/.style={circle,draw,fill=red,thick,
                   inner sep=0pt,minimum size= .8 mm},
                  
      trans/.style={thick,->, shorten >=6pt,shorten <=6pt,>=stealth},
   ]

    \node[vertexr] (v_{1}) at (0, 0) {};
    \node[vertex] (v_{2}) at (1, 0) {};
    \node[vertex] (v_{3}) at (2, 0) {};
    \node[vertex] (v_{4}) at (3, 0) {};
    \node[vertex] (v_{5}) at (4, 0) {};
    \node[vertex] (v_{6}) at (5, 0) {};
    \node[vertex] (v_{7}) at (6, 0) {};
    \node[vertex] (v_{8}) at (7, 0) {};
    \node[vertex] (v_{9}) at (8, 0) {};
    \node[vertex] (v_{10}) at (7, -1) {};
\draw [dash dot](v_{1})--(v_{2}){};
\draw [->](v_{2})--(v_{3}){};
\draw [->](v_{3})--(v_{4}){};
\draw [->](v_{4})--(v_{5}){};
\draw [->](v_{5})--(v_{6}){};
\draw [->](v_{6})--(v_{7}){};
\draw [->](v_{7})--(v_{8}){};
\draw [->](v_{8})--(v_{9}){};
\draw[->, dash dot](v_{10})-- (v_{9}){};

\node at (0.5, 0.2){$a$};
\node at (1.5, 0.25){$b$};
\node at (2.5, 0.2){$a$};
\node at (3.5, 0.25){$b$};
\node at (4.5, 0.2){$a$};
\node at (5.5, 0.25){$b$};
\node at (6.5, 0.2){$a$};
\node at (7.5, 0.25){$b$};
\node at (8.5, -0.6){$c(ab)^{4}$};
\node at (-.4, 0){$\BP$};
\node at (6.8, -1.3){$\overline c \, \BP$};

\draw[dash dot] (v_{1})--(v_{2}){};
\end{tikzpicture}
\caption{The slice $C_{e_{a}}$ over $e_{a} \in T_{0}$ is a path of length 7.}
\end{figure}
The whole slice lies on $\axis_T(ab)$, therefore
\[
 \frac{|\axis_T(ab) \cap C_{e_{a}}(T)|}{\ell_T (\alpha)} = \frac 72 =3.5.
\] 
Meanwhile, if we consider $e_{b}=({\rm id}, b)_T$, the slice over $e_{b}$ is the 
set of solid edges in the picture below. 
\begin{figure}[h!]
\begin{tikzpicture}[scale=0.8]
 \tikzstyle{vertex} =[circle,draw,fill=black,thick, inner sep=0pt,minimum size=.5 mm]
 
[thick, 
    scale=1,
    vertex/.style={circle,draw,fill=black,thick,
                   inner sep=0pt,minimum size= .5 mm},
                  
      trans/.style={thick,->, shorten >=6pt,shorten <=6pt,>=stealth},
   ]
   
    \tikzstyle{vertexr} =[circle,draw,fill=red,thick, inner sep=0pt,minimum size=.8 mm]
 
[thick, 
    scale=1,
    vertex/.style={circle,draw,fill=red,thick,
                   inner sep=0pt,minimum size= .8 mm},
                  
      trans/.style={thick,->, shorten >=6pt,shorten <=6pt,>=stealth},
   ]

\node[vertexr] (v_{2}) at (1, 0) {};
\node[vertex] (v_{3}) at (2, 0) {};
\node[vertex] (v_{4}) at (3, 0) {};
\node[vertex] (v_{5}) at (4, 0) {};
\node[vertex] (v_{6}) at (5, 0) {};
\node[vertex] (v_{7}) at (6, 0) {};
\node[vertex] (v_{8}) at (7, 0) {};
\node[vertex] (v_{9}) at (8, 0) {};
\node[vertex] (v_{10}) at (7, -1) {};
\draw [dash dot](v_{2})--(v_{3}){};
\draw [->](v_{3})--(v_{4}){};
\draw [->](v_{4})--(v_{5}){};
\draw [->](v_{5})--(v_{6}){};
\draw [->](v_{6})--(v_{7}){};
\draw [->](v_{7})--(v_{8}){};
\draw [->](v_{8})--(v_{9}){};
\draw[->, dash dot](v_{10})-- (v_{9}){};

\node at (1.5, 0.25){$b$};
\node at (2.5, 0.2){$a$};
\node at (3.5, 0.25){$b$};
\node at (4.5, 0.2){$a$};
\node at (5.5, 0.25){$b$};
\node at (6.5, 0.2){$a$};
\node at (7.5, 0.25){$b$};
\node at (8.5, -0.6){$c(ab)^4$};
\node at (.6, 0){$\BP$};
\node at (6.8, -1.3){$\overline a \,  \overline c \, \BP$};

\draw[dash dot, thick] (v_{2})--(v_{3}){};
\end{tikzpicture}
\caption{The slice $C_{e_{b}}$ over $e_{b} \in T_{0}$ is a path of length 6.}
\end{figure}

Since this slice also lies on $\axis_T(ab)$, we have 
\[
 \frac{|\axis_T(ab) \cap C_{e_{b}}(T)|}{\ell_T (\alpha)} = \frac 62 =3.
\] 
Which means, if one starts from $b$, one observes less twisting. 
Lastly, the edge $e_c=({\rm id}, c)_T$ has only one pre-image in $T$. 
This implies that $C_{e_{c}} = \emptyset$ (again, see \cite[Lemma 3.7]{uniformsurgery}).
Thus the estimate for the twist depends heavily on the choice of $e_0$. In our example, 
\[
\twist_{\alpha}(T, T_0) = \max \Big\{ 3.5, 3, 0 \Big\} = 3.5
\qquad\text{and}\qquad
\tw_\alpha(T, T_0) = [3.5]= 3. 
\]
\end{example}

\begin{remark}\label{ClayandPettet}
The twisting number defined above is a rational number. The integer part of 
$\twist_{\alpha}(T, T_0)$, which we denote by $\tw_{\alpha}(T, T_0)$, 
is equal to the Clay-Pettet definition of relative twisting number which
considers hyperplanes $\tau, \tau_{0}$ in the core that are dual to specific edges
and counts how many $\alpha$--translates of $\tau$ intersects $\tau_{0}$
(see \cite{twistdefinition} for details). Similar to the above example, one can see that 
in the example of Theorem A,
\[
\tw_{b} (R_{0}, \phi(R_{0}))= s-1 \qquad \text {and}\qquad 
\tw_{ab^{s}}(R_{0}, \phi(R_{0})) = t-1.
\]
\end{remark}

We now show that the relative twisting number changes slowly along loops with large 
lengths. For a real number $r>0$, let $[r]$ be the integer part of $r$ and 
$\{ r\}$ be the fractional part of $r$. 

\begin{theorem}\label{slowchange}
Let $R'$ be obtained from $R$ by a single fold:
\[
R' = \fold(R, t, s)
\]
Then, for any loop  $\alpha$,
\begin{equation}\label{Eq:all cases}
\tw_\alpha(R', R_0) \geq \tw_\alpha(R, R_0)-1
\end{equation}
where $R_0$ is as defined before. Furthermore, 

\begin{equation}\label{Eq:fractional}
\twist_\alpha(R', R_0) \geq 
\left[ \twist_\alpha(R, R_0)- \frac {2}{\ell_{R}(\alpha)} \right]
+ \frac{\left\{ \twist_\alpha(R, R_0 )- \frac  {2}{\ell_{R}(\alpha)} \right\}}4.
\end{equation}
\end{theorem}

We need to prepare for the proof by establishing a few lemmas. 
Note that $R$ can also be obtained from $R'$
by a single fold: 
\[
R' = \fold(R, t, s), \qquad R= \fold(R', \os t, \os)
\]
For the rest of this section, we assume $T$ and $T'$ are universal covers of $R$ and $R'$
respectively and that
\[
f \from T \to T' 
\qquad \text{and} \qquad 
g \from T' \to T
\]
are the morphism associated to these folds.  For an embedded edge path
$P = [e_1 \dots e_k]$ in $T$, let $f(P)$ denote the image of $P$ under the morphism 
$f$ and let $f(P)_w$ denote the embedded edge path that is the intersection of $f(P)$ 
and $\axis_{T'}(w)$. We call $e_1$ and $e_k$ the \emph{end edges} of $P$.

\begin{lemma}\label{factorof2}
For any loop $\alpha$, we have
\[
\ell_T(\alpha) \geq \frac{\ell_{T'}(\alpha)}{2}. 
\]
Also, if $P= [e_1 \dots e_k]$ is an edge path on the $\axis_T(w)$, for some $w \in \F$, 
so that both $f(e_1)$ and $f(e_k)$ contain an edge on $\axis_{T'}(w)$, then 
\[
\big| f(P)_w \big| \geq \frac{|P|}2.
\]
\end{lemma}

\begin{proof}
Recall that $f$ maps an edge $(u, u t)_T$ to the edge path
\[ 
\big[ (u, us)_{T'}, (us, us(\overline{s} t))_{T'} \big]
\]
and maps every other edge to one edge. Therefore, for any embedded
edge path $P$ in $T$, 
\begin{equation}\label{twice}
 \big| f(P) \big| \leq 2 |P|.
 \end{equation}
If $P$ is an edge path that realizes $\ell_{T}(\alpha)$, then 
$\big| f(P) \big| \geq \ell_{T'}(\alpha)$, thus
\[
\ell_T(\alpha)  =  |P| \geq \frac {\big| f(P)\big|}{2} \geq \frac{\ell_{T'}(\alpha)}{2}. 
\]

Now assume $P= [e_1 \dots e_k]$ is an edge path on the $\axis_T(w)$ and
both $f(e_1)$ and $f(e_k)$ contain an edge on $\axis_{T'}(w)$. 
Applying, Equation~\ref{twice} to the morphism $g$ and the edge path 
$f(P)_w$, we have 
\begin{equation}\label{wednesday}
 \big| g(f(P)_w)\big| \leq 2 \big| f(P)_w \big|. 
 \end{equation}
We need to show $P \subseteq g(f(P)_w)$. In fact, it suffices to show
that any end vertex of $P$ is contained in $g(f(P)_w)$. 

Let $u_T$ be the first vertex in $e_1$ (the argument for the last vertex of $e_k$ 
is similar). If $u_{T'} \in \axis_{T'}(w)$, then $u_{T'} \in f(P)_w$, which means 
$u_T \in g(f(P)_w)$. Otherwise,   $u_{T'} \notin \axis_{T'}(w)$. By assumption, $f(e_1)$ 
contains an edge on $\axis_{T'}(w)$. It follows that $e_1$ is mapped to two edges 
$[e', e'']$ under $f$, which means $e_1$ is necessarily a $t$--edge (or a $t^{-1}$--edge, in which the remainder of the proof changes accordingly). 
Furthermore, the edge $e'$ has two pre-images. This is because, 
$u_{T'} \notin \axis_{T'}(w)$ and the edge preceding $e_1$ along the 
$\axis_T(w)$ must be mapped over $e'$ as well. But the edges in $T'$ with label 
$\os t$ have only one pre-image. Hence, 
\[
e_1 = (u ,ut)_T, \qquad e'=(u, us)_{T'}, \qquad\text{and} \qquad e''=(us, ut)_{T'}.
\] 
We then have
\[
(us)_{T'}, (ut)_{T'} \in \axis_{T'}(w) \qquad\Longrightarrow\qquad 
(us)_T, (ut)_T \in g\big( f(P)_w \big). 
\]
But $T$ is a tree and the vertex $u_T$ necessarily lies on the path connecting 
$(us)_T$ and $(ut)_T$. Thus $(us)_T, (ut)_T \in g(f(P)_w)$ implies $u_T \in g(f(P)_w)$. 

We have shown that the end vertices of $P$ are both in $g(f(P)_w)$, 
which implies $P \subseteq g(f(P)_w)$. The lemma follows from 
\eqnref{wednesday}. 
\end{proof}

Let $\alpha \in \F$ be a loop. Let $e_0 \in T_0$ and $w \in \alpha$ be so that 
the edge path 
$$
P := [e_1e_2\dots e_k] = \axis_T(w) \cap C_{e_0}(T) 
$$
is the one realizing the maximum in the definition of $\twist_\alpha(T, T_0)$, although we remark that maximality is not needed for the following
result.
\begin{lemma}\label{threecases}
Let $P$ be the path above and assume $|P| \geq 2$. Then $f([e_1e_2])$ contains 
an edge in $C_{e_0}(T')\cap \axis_{T'}(w)$. The same holds for $f([e_{k-1}e_k])$. 
Furthermore, if $\ell_T(\alpha)=1$, then either $f(e_1)$ or $f(e_k)$
contains an edge in $C_{e_0}(T')\cap \axis_{T'}(w)$. 
\end{lemma}

\begin{proof}
The edge path $f(P)$ contains $f(P)_w$. Furthermore, an edge in $f(P)$ is on $f(P)_w$
if and only if it has a unique pre-image in $\axis_T(w)$. 

If, for $i=1$ or $2$, $e_i=(u, v)$ and $\overline{u}v$ is not $s$, $\overline{s}$, $t$ or 
$\overline{t}$, then $f(e_i)$ is a single edge, has one pre-image 
and it is boundary equivalent to $e_i$. Hence, it lies on $C_{e_0}(T')\cap  \axis_{T'}(w)$.

If $e_i=(u, ut)$, by the definition of $\fold(T, t, s)$,
$e_i$ is mapped to the edge path 
\[ 
[(u, us)(us, us(\overline{s} t))].
\] 
Since the edge $(us, us(\overline{s} t))$ has only one pre-image, it lies
on $f(P)_w$ and by Proposition~\ref{boundaryequivalent}, it is 
boundary equivalent to $e_i$ and thus $(us, us(\overline{s} t))\in C_{e_0}(T')$.
Similar argument works when $e_i= (ut, u)$. Also, this shows that if
$\alpha = t$ or $\overline t$, then either $f(e_1)$ or $f(e_k)$
contains an edge in $C_{e_0}(T')\cap \axis_{T'}(w)$.

There are two remaining cases. Assume $e_1=(u, us)$ and $e_2 = (us, us^2)$. 
Then the $t$--edge starting at $us$, $(us, ust)$, does not lie on $\axis_T(w)$. 
Hence, $f(e_2)$ has one pre-image in $\axis_T(w)$ and thus is on $\axis_{T'}(w)$. 
Also, by part (iii) of \propref{boundaryequivalent}, 
\[
\partial^+(e_2) \subset \partial^+(f(e_2)) 
\]
and by part (iv) of \propref{boundaryequivalent}, 
\[
\partial^-(e_1) \subset \partial^-(f(e_2)).
\]
Since $e_1, e_2 \in C_{e_{0}}(T)$, each of $\partial^+(e_2)$ and $\partial^-(e_1)$ intersects 
each of $\partial^+(e_0)$ and $\partial^-(e_0)$. Therefore, each of $\partial^+(f(e_2))$ 
and $\partial^-(f(e_2))$ intersects each of  $\partial^+(e_0)$ and $\partial^-(e_0)$. That is, 
$f(e_2) \in C_{e_{0}}(T')$. 

The remaining case when $e_1=(us^2, us)$ and $e_2 = (us, u)$ is identical, except
in this case, $f(e_1)$ is in $C_{e_0}(T')\cap \axis_{T'}(w)$. 
These two cases also show that if $\alpha = s$ or $\overline s$, then both $f(e_1)$ 
and $f(e_k)$ contains an edge in $C_{e_0}(T')\cap \axis_{T'}(w)$. 
\end{proof}

\begin{proof}[Proof of Theorem~\ref{slowchange}]
Recall $\alpha \in \F$ is a loop and 
the edge path 
$$
P := [e_1e_2\dots e_k] = \axis_T(w) \cap C_{e_0}(T) 
$$
is the edge path that realizes the maximum in the definition of $\twist_\alpha(T, T_0)$. 
By Lemma~\ref{threecases}, either $f(e_1)$ or $f(e_2)$ contains an edge that is in 
$C_{e_0}(T') \cap \axis_{T'}(w)$. Call the associated edge in $T$ (either $e_1$ or $e_2$), 
$e_{\rm first}$. Likewise, one of $f(e_{k-1})$ or $f(e_k)$ has this property.
Call the associated edge in $T$ (either $e_{k-1}$ or $e_k$) $e_{\rm last}$. 
The combinatorial length of the edge path $[e_{\rm first} \dots e_{\rm last}]$ is 
at least 
\[
k-2= \twist_{\alpha}(T, T_0) \, \ell_T(\alpha) -2.
\]
Hence, if we define
\begin{equation} \label{Eq:Integer}
p := \left[ \twist_{\alpha}(T, T_0)-\frac  {2}{\ell_{R}(\alpha)} \right]\\
\end{equation}
then, $(k-2) \geq p \cdot \ell_R(\alpha) = p \cdot \ell_T(\alpha)$.
Therefore, $[e_{\rm first} \dots e_{\rm last}]$ contains at least $p$ copies of the 
fundamental domain of the action of $w$ on $\axis_T(w)$ and, in particular, 
$w^p(e_{\rm first})$ (the translation of $e_{\rm first}$ by the action of $w^p$)
lies on $[e_{\rm first} \dots e_{\rm last}]$. Let $P_{\rm remain}$ be the edge path 
$[w^p(e_{\rm first}) \dots e_{\rm last}]$. Since $e_{\rm first}$ contains
an edge in $\axis_{T'}(w)$, so does $w^p(e_{\rm first})$. That is, 
$f(P)_w$ consists of $p$ fundamental domains of the action of $w$
on $\axis_{T'}(w)$ and the segment $f(P_{\rm remain})_w$. Thus
\begin{equation} \label{Eq:twist_after}
\twist_\alpha(T', T_0) \geq 
\frac{p \ell_{T'}(\alpha) +|f(P_{\rm remain})_w|_{T'}}{\ell_{T'}(\alpha)}=
p + \frac{|f(P_{\rm remain})_w|_{T'}}{\ell_{T'}(\alpha)}.
\end{equation}
Also, by \lemref{factorof2}
\[
\ell_{T'}(\alpha) \leq 2 \ell_{T}(\alpha)
\qquad\text{and}\qquad
|f(P_{\rm remain})_w|_{T'} \geq \frac{|P_{\rm remain}|_T}2. 
\]
Also, since 
\[
|P_{\rm remain}|_T \geq |P|_T - p \, \ell_T(\alpha) \geq (k-2) - p \, \ell_T(\alpha), 
\] 
we have 
\begin{align*}
\frac{ |P_{\rm remain}|_T}{\ell_T(\alpha)} &\geq
\frac{k}{\ell_T(\alpha)} - \frac 2{\ell_T(\alpha) } - p \\
\end{align*}
Let $L: = \ell_T(\alpha)$, we then have
\begin{align*}
\frac{ |P_{\rm remain}|_T}{\ell_T(\alpha)} &\geq
\frac{k}{\ell_T(\alpha)} - \frac 2{\ell_T(\alpha) } - p \\
 & \geq \twist_{\alpha}(T, T_0)-\frac 2L - p = 
\left\{ \twist_{\alpha}(T, T_0)-\frac 2L \right\},
\end{align*} 
where the last equality follows from the definition of $p$. Hence
\begin{align} \label{Eq:Remain}
 \frac{|f(P_{\rm remain})_w|_{T'}}{\ell_{T'}(\alpha)} 
 \geq \frac{|P_{\rm remain}|_{T}/2}{2\ell_{T}(\alpha)} 
 \geq \frac{\left\{ \twist_{\alpha}(T, T_0)-\frac 2L \right\}}4. 
\end{align}
Equation~\ref{Eq:fractional} follows from Equations \eqref{Eq:Integer}, \eqref{Eq:twist_after}and \eqref{Eq:Remain}. 

\medskip
For $L \geq 2$, Equation~\ref{Eq:fractional} implies Equation~\ref{Eq:all cases}. 
If $L =1$, from the second assertion of \lemref{threecases}, we have that 
either $e_1 = e_{\rm first}$ or $e_k = e_{\rm last}$. Hence, $f(P)_w$ contains at least 
$(k-1)$ fundamental domains of action of $w$ and $\tw_\alpha(T', T_0)\geq k-1$. 
But $\tw_\alpha(T, T_0)=k$. Therefore, Equation~\ref{Eq:all cases} still holds. 
\end{proof}

\begin{theorem}\label{Thm:slowtwist}
For any folding sequence $T_m, \dots, T_0$ and any loop $\alpha$, we have
\[
m \geq \tw_\alpha(T_m, \, T_0).
\] 
Further, if $\ell_{T_i}(\alpha)\geq L>50$ for every $i$, then 
\[
m > \tw_\alpha (T_m, \, T_0) \left(\log_{5} \frac{L}{50}\right). 
\] 
\end{theorem}

\begin{proof}
The first assertion of the theorem follows directly from the first assertion of 
\thmref{slowchange}. We prove the second assertion. 

For a real number $\frac{10}{L}< r$, we have
\begin{equation} \label{Eq:5} 
r-\frac 2L> \frac{4r}{5} 
\qquad\text{and}\qquad
\frac{r -\frac2L}{4} >  \frac{r}{5}.
\end{equation}

For any $0\leq N < \tw_\alpha(T_m, T_0)$, consider the first time $i=i(N)$ 
when $\tw_\alpha(T_i, T_0) = N$. Since the index decreases as we fold towards $T_{0}$, $i$ is the maximal of all the times $\tw_\alpha(T_j, T_0) = N$. 
Now apply \thmref{slowchange}  with $T = T_{i+1}$ and $T' = T_{i}$. Since $i$ is the first time $\tw_\alpha(T_i, T_0) = N$, we have 
\[
[\twist(T_i, T_0)] = N \qquad \text{and} \qquad [\twist(T_{i+1}, T_{0}) ]= N+1.
\]
That is to say, $[\twist(T_{i+1}, T_{0}) - \frac{2}{L}] \geq N$ as $\tfrac 2L < 1$. 
By Equation~\ref{Eq:fractional}, 
\begin{align*}
\twist(T_i, T_0) &\geq \Big[\twist(T_{i+1}, T_{0})- \frac{2}{L} \Big] + \frac{\Big\{\twist(T_{i+1}, T_{0})- \frac{2}{L}\Big\}}4\\
\intertext{therefore}
N+ \{\twist(T_i, T_0) \}  
  &\geq N + \frac{\Big\{\twist(T_{i+1}, T_{0})- \frac{2}{L}\Big\}}4\\
\intertext{and hence}
 \{\twist(T_i, T_0) \}  &\geq  \frac{\Big\{\twist(T_{i+1}, T_{0})- \frac{2}{L}\Big\}}4.
\end{align*}

Again since this is the first time that $\tw_\alpha(T_i, T_0) = N$, we have
\[ \twist(T_{i+1}, T_0) \gneq N \geq 0. \]
That is to say $\twist(T_{i+1}, T_0) \geq 1$. From \eqnref{Eq:fractional}, we also have 
\[
\twist_\alpha(T_{i}, T_0) \geq 
\left[ \twist_\alpha(T_{i+1}, T_0)- \frac {2}{L} \right].
\]
Combined with the fact that $ \tw_\alpha(T_{i}, T_0)  < \tw_\alpha(T_{i+1}, T_0)$, we have that 
\begin{equation}\label{fractionalreduce}
\Big \{\twist_\alpha(T_{i+1}, T_0) \Big \} <  \frac {2}{L}.
\end{equation}
Therefore, 
\begin{align*}
\Big\{\twist(T_i, T_0)\Big\} &\geq \frac{\Big\{\twist(T_{i+1}, T_{0})- \frac{2}{L}\Big\}}4 \\
& = \frac{\Big\{\twist(T_{i+1}, T_{0}) + 1- \frac{2}{L}\Big\}}4 \\
& = \frac{\Big\{ \big\{\twist(T_{i+1}, T_{0})  \big \}+ 1- \frac{2}{L}\Big\}}4 \\
&\geq  \frac{\Big\{ 1- \frac{2}{L}\Big\}}4
 =\frac{1- \frac{2}{L}}4 \geq \frac 15.
\end{align*}

Let $(i-k-2)$ be first time when $\tw_\alpha(T_{i-k-1}, T_0) = N-1$. Then the step immediately before that we again have
\[
\Big\{\twist(T_{i-k-1}, T_0)\Big\} \leq \frac{2}{L}. 
\]
Which means for the step before that we have 
\[
\frac{\Big\{\twist(T_{i-k}, T_{0})- \frac{2}{L}\Big\}}4 \leq \Big\{\twist(T_{i-k-1}, T_0)\Big\}  \leq \frac{2}{L}. 
\]

Which means  $\Big\{\twist(T_{i-k}, T_{0}) \Big\} \leq \frac {10}{L}$.
That is, in $k$-steps, the fractional twist has been reduced from above $\frac 15$ 
to below $\frac {10}{L}$. Thus, from  \propref{slowchange} and \eqnref{Eq:5}, we have
\[
\frac{\frac 15}{5^k} < \frac{10}{L},\]
which can be rewritten as
\[k > \log_5 \frac L{50}. 
\]
Since this is true for every $0\leq N < \tw_\alpha(T_0, T_m)$, we have
\[ 
 \tw_\alpha (T_m, T_0) <  \frac{m}{\log_{5} \frac{L}{50}}\\
\quad\Longrightarrow\quad 
m > \tw_\alpha (T_m, T_0)  \left(\log_{5} \frac{L}{50}\right), 
\]
which is as desired. 
 \end{proof}

\subsection{ The lower bound is sharp}
At first glance, one might think, that the factor $\log_{5} \frac{L}{50}$  in \thmref{slowchange}
could be replaced with a linear function of $L$. However, we show that the above 
estimate is sharp up to a uniform multiplicative error. We now construct, for an 
arbitrarily large $L$  a folding path $R=R_m, \dots, R_0$ so that
\begin{itemize}
\item The length of $\alpha$ at each $R_i$ is at least $L$. 
\item $m$ is comparable with $\log L \tw_\alpha(R, R_0)$. 
\end{itemize} 

\begin{example} \label{Exa:Log}
Let $R$ be a rose of rank $5$ with edge labels 
\[
\big\langle (bc)^m a, d \, b , \phi^k(d) c, d, e \big\rangle
\]
where $\phi \from \langle d, e \rangle \to \langle d, e \rangle$ is a
fully irreducible automorphism of the free factor $\langle d, e \rangle$
with exponential growth and 
\[ \lambda(\phi) = \lambda(\phi^{-1}).\]
Let $L$ be the word length of $\phi^{\lfloor k/2 \rfloor}(d)$
(and also the length of $\phi^{-\lfloor k/2 \rfloor}(d)$) in $\langle d, e \rangle$
which can be chosen to be arbitrarily large by choosing $k$ large enough. Let 
$\alpha = [bc]$.  Let $R_0$ be a rose with labels 
\[
\langle a, d \, b , \phi^k(d) c, d ,e \rangle.
\]

The shortest way to express $\alpha$ in $R$ is 
\[
bc = (\bd)  \cdot (db) \cdot \phi^{-k}(d) \cdot (\phi^k(d) c),
\]
where the terms in parentheses are labels of edges in $R$ and is a word of 
length roughly $L^2$ in $\langle d, e \rangle$. We have
\[
\ell_{R} (\alpha) \asymp L^2 > L. \]
Similar to Example~\ref{twistingexample}, we have 
\[
 \tw_{\alpha}(R, R_0) = m-1.
\]
We now start twisting around the loop $[bc]$, however in a somewhat 
un-natural way that always keeps the length of $\alpha$ larger than $L$, using 
the following steps:
\begin{enumerate}[1)]
 \item Twist the $((bc)^m a)$--loop around the first half of $\alpha$, that is to say, cancel $b$ 
 (which is half of $\alpha$) from $(bc)^m a$. To do this we first fold the $((bc)^m a)$--loop around the $(d)$--loop, and then fold the resulting loop around $(db)$--loop. Thus this step uses 2 folds: 
 
\[\overline{(db)} \cdot d \cdot (bc)^m a = c (bc)^{m-1}a.\]
 \item  Fold $\langle d, e \rangle$ to $\langle \phi^{k}(d), \phi^k(e)\rangle$; this takes 
\[
k \, \Vert \phi \Vert_{\langle d, e \rangle} \asymp \log L
\]
many steps. Note that the immersed loop
 \[
bc = (\bd)  \cdot (db) \cdot \phi^{-k}(d) \cdot (\phi^k(d) c),
\]
contains both $\bd$ and $\phi^{-k}(d)$. At any point along this folding path, 
say after $\phi_i$ has been applied $i$ times, we have
\[
\max\Big( \big| \bd \big|_{\big\langle \phi^i (d), \phi^i (e) \big\rangle},  
   \big| \phi^{-k}(d) \big|_{\big\langle \phi^i (d), \phi^i (e) \big\rangle}\Big) 
  \geq |\phi^{\lfloor k/2 \rfloor}(d)| = L. 
\]
where $| \param |_{\big\langle \phi^i (d), \phi^i (e) \big\rangle}$ denotes the word
length of an element in the group $\langle d,e \rangle$ in terms of $\phi^i (d)$ and
$\phi^i (e)$. Thus the combinatorial lengths of $\alpha$ remains at least $L$ for each $i$.
 
\item Twist around the second half of $\alpha$, that is, cancel $c$ from 
$c (bc)^{m-1}a$ in 2 folds:
  \[ \overline{ (\phi^k(d) c)} \cdot \phi^{k}(d) \cdot c(bc)^{m-1}a = (bc)^{m-1}a.\]
\item Fold $\langle \phi^{k}(d), \phi^k(e)\rangle$ to $\langle d, e \rangle$. Again, 
the number of steps is comparable to $\log L$ and the length of $\alpha$ remains
larger than $L$. 
\end{enumerate}

We now repeat steps 1--4, $m$--times. Every time the relative twisting around 
$\alpha$ is reduced by 1. The path has a length of order $m \log L$
as desired. 
\end{example}

\section{Quasi-Geodesics in \Out} \label{Sec:Quasi} 
In this section we use \thmref{Intro:slowtwist} to prove \thmref{Intro:backtrack}. 
Consider a path $p \from [0,m] \to \calX$ from an interval $[0,m] \subset \RR$ to 
a metric space $\calX$ and, for $i \in [0,m]$, let $p_i = p(i)$. Recall that $p$ is a 
$(K,C)$--quasi-geodesic if for any 
$i, j \in [0, m]$, we have
\[
\frac {|i-j| - C}{K} \leq d_\calX\big(p_i, p_j\big) \leq K \, |i-j| +C. 
\]
It follows that, any 3 points in the image of a quasi-geodesic satisfy 
a coarse reverse triangle inequality. Namely, for any $i, j, k \in [0,m]$,
$i \leq j \leq k$, we have 
\begin{equation} \label{Eq:Reverse} 
d_\calX(p_i, p_j) + d_\calX(p_j, p_k) \leq K^2 d_\calX(p_i, p_k) + (K+2)C. 
\end{equation} 
We say $p$ is a re-parametrized $(K,C)$--quasi-geodesic if there is a 
re-parametrization $\rho \from [0, m] \to [0,m]$ so that $p \circ \rho$ is a 
$(K,C)$--quasi-geodesic. Since the image of $p$ and $p \circ \rho$
are the same, if $p$ is a re-parametrized quasi-geodesic, any 3 points 
in its image still satisfy the coarse-reverse-triangle inequality given in \eqnref{Eq:Reverse}. 
Often, it is convenient to consider maps from intervals $[0,m]_\ZZ$ in $\ZZ$ to a metric
space $\calX$. Then we say $p \from [0,m]_\ZZ \to \calX$ is a re-parametrized 
quasi-geodesic if it is a restriction of a re-paramterized quasi-geodesic 
from $[0,m] \to \calX$. 

We can now restate \thmref{Intro:backtrack} explicitly as follows: 

\begin{theorem}[Quasi-geodesics back-track in sub-factors] \label{Thm:Stallings-bad}
For given constants $K_1$, $C_1$, $K_2$, $C_2 >0 $ there exists an automorphism, 
$\phi \in \Out$ such that, for any $(K_1, C_1)$--quasi-geodesic $p \from [0,m] \to \Out$ 
with $p(0) = id$ and $p(m) = \phi$, the shadow $\Theta_\calA \circ p$ of $p$ 
in $\calF(A)$ is not a $(K_2, C_2)$--reparameterized quasi-geodesic.
\end{theorem}

\begin{proof}
Let $ \langle a, b, c \rangle $ be a generating set for $\mathbb{F}_3$, and let 
$R_0$ be a rose with labels $\{a, b, c \}$. Let $\psi \in \out(\mathbb{F}_3)$ be 
an automorphism defined as:
\[
 \left \{
 \begin{array}{ll}a \longrightarrow aba  \\
  b \longrightarrow ab\\
  c \longrightarrow c
\end{array}
\right.
\]

Given the generating set introduced in the introduction,
$\Vert \psi \Vert = \Vert \psi^{-1} \Vert= 2$ where $\Vert \param \Vert$ represents the 
word length. Let $\calA = \langle a, b \rangle < \calF_3$ be a rank $2$ free factor.  
The automorphism $\psi$ fixes $\calA$. We denote the restriction of $\psi$
to $\calA$ by $\psi_{\calA}$. Then, $\psi_{\calA}$ is an irreducible automorphism 
and acts loxodromicly on the free-factor graph $\calF(\calA)$ of $\calA$. 
Hence, there exists a constant $c_\psi>0$ so that, for an integer $q>0$
\begin{equation} \label{Eq:lox} 
d_\calA\big(R_0, \psi^q(R_0)\big) \geq c_\psi q. 
\end{equation}
Let $\alpha$ be the loop represented by the word $\psi^q(a)$ and, for a 
large integer $t>0$, let $\phi \in \out(\mathbb{F}_3)$ be the automorphism that twists the element $c$ 
around $\alpha$ $t$--times, namely:
\[
 \left \{
 \begin{array}{ll}a \longrightarrow a  \\
  b \longrightarrow b\\
  c \longrightarrow c(\psi^q(a))^t
\end{array}
\right.
\]

We first find an upper-bound for the $\Vert \phi \Vert$ by constructing a path connecting
identity to $\phi$. First apply $\psi^q$ so $\alpha$ is represented by one edge in the
rose $\psi^q(R_0)$, then twist $c$ around $\alpha$ $t$--times, and then apply $\psi^{-q}$. 
We have
\[
\Vert \phi \Vert \leq q(\Vert \psi \Vert + \Vert \psi^{-1} \Vert) + t = 4q+t. 
\]

Now consider the $(K_1, C_1)$--quasi-geodesic $p \from [0,m] \to  \out(\mathbb{F}_3)$
connecting the identity to $\phi$. We have
\begin{equation} \label{Eq:m}
m \leq K_1 \Vert \phi \Vert + C_1 \leq 4K_1 \, q + K_1 \, t + C_1. 
\end{equation}
If the shadow of $p$ to $\calA$ is a $(K_2, C_2)$--reparatmetrized quasi-geodesic, then the coarse-reverse-triangle-inequality (\eqnref{Eq:Reverse}) holds. That is, for any 
index $i$ and $R_i = p(i) (R_0)$, we have
\[
d_\calA(R_0, R_i) + d_\calA(R_i, \phi(R_0)) \leq 
(K_2)^2 d_\calA(R_0, \phi(R_0)) + (K_2+2)C_2. 
\]
But $\phi$ fixes $a$ and $b$ and thus $R_0$ and $\phi(R_0)$ have the same projection to 
$\calA$. Hence, 
\[
2 d_\calA(R_0, R_i)  \leq (K_2+2)C_2. 
\]
Now, using \eqnref{Eq:lox}, we get
\[
d_\calA(R_i, \psi^q(R_0)) \geq d_\calA(R_0, \psi^q(R_0)) - d_\calA(R_0, R_i) 
   \geq c_\psi q - (K_2+2)C_2. 
\]
By \lemref{Laddy}, this implies
\begin{equation}
\ell_{R_i}(\alpha) \geq \frac{c_\psi \, q -(K_2+2)C_2-13}{6} =: L 
\end{equation} 
Now, \thmref{Thm:slowtwist} implies that, 
\[
t = \tw_\alpha(R_0, R_m) \leq \frac{m}{\log_{5} \frac{L}{50}}
\]
and using \eqnref{Eq:m} we get, 
\[
t \leq \frac{4K_1 \, q + K_1 \, t + C_1}{\log_{5} \frac{L}{50}}. 
\]
If we choose $q$ large enough so that 
\[\log_{5} \frac{L}{50} > 2K_1\]
and then choose $t$ large enough so that 
\[\frac t2 >  \frac{4 K_1 \, q + C_1}{\log_{5} \frac{L}{50}}
\]
we get  
\[
t \leq \frac{4K_1 \, q + C_1}{\log_{5} \frac{L}{50}}
+ \frac{K_1 \, t }{\log_{5} \frac{L}{50}} <
\frac t2 + \frac t2 = t
\]
which is a contradiction. The contradiction proves that the shadow of $p$ to 
$\calF(\calA)$ is not a re-parametrized $(K_2, C_2)$--quasi-geodesic. 
\end{proof}

\section{Outer Space} \label{Sec:Outer-Space}
Outer Space $\CV$ is metric space with $\Out$ action defined as an analogue 
of the \Teich space. See \cite{outerspace} for more details. 
Here, we introduce $\CV$ briefly and prove Theorem D. 

We assume a \emph{graph} is always simple and all vertices have degree 3 or more.
A \emph{marked metric graph} $(x, f)$ is a metric graph $x$ together with homotopy 
equivalence $f \from R_0 \to x$. The space of all marked metric graphs whose edge 
lengths sum up to one is called the \emph{Outer Space}~\cite{outerspace} and 
is denoted by $\CV$. The group $\Out$ acts on $\CV$ by precomposing the marking: 
for an element $\phi \in \Out$, $\phi(x, f) = (x, f\circ \phi) $. 

Note that we can still think of $x$ as a labeled graph. Recall that $\ell_x(\alpha)$ denotes 
the combinatorial length of $\alpha$ in $x$. Let $|e|_x$ denote the metric length 
of an edge $e$ in $x$ and $|\alpha|_x$ the metric length of $\alpha$ in $x$, which is 
the metric length of the immersed loop of the representative of $\alpha$ that realizes 
its combinatorial length in $x$. 

For a fixed $\epsilon > 0$ define the \emph{thick part} of $\CV$ to be the set of $x \in \CV$ such that 
\[
|\alpha|_x \geq \epsilon \text{ for every nontrivial conjugacy class $\alpha$. }
\]
 
A map $h \from (x, f_x) \to (y, f_y)$ is a \emph{difference of markings} map if 
$h \circ f_x \simeq f_y$ (homotopy). We will only consider Lipschitz maps and we 
denote by $L_h$ the Lipschitz constant of $h$.  In many ways it is natural 
to consider the (asymmetric) Lipschitz metric on $\CV$:
$$
d(x, y):=\inf_h \log L_h
$$
where the infimum is taken over all differences of markings maps. We refer the reader 
to \cite{FMmetric, asymmetry} for review for some metric properties of 
$d(\param, \param)$. In particular, there always exists a non-unique 
difference of markings map that realizes the infimum. Since a difference of markings 
map is homotopic rel vertices to a map that is linear 
on edges, we also use $h$ to denote the representative that realizes the infimum and is 
linear on edges and refer to such a map as an \emph{optimal map} from
$x$ to $y$. For this section we always assume $h$ is an optimal difference of markings
map. Since $h$ is linear on edges, we define 
\[
\lambda(e) =  \frac{|h(e)|_y}{|e|_x}
\]
to be the \emph{stretch factor} of an edge $e$ and 
\[
\lambda(\alpha) =  \frac{|\alpha|_y}{|\alpha|_x}
\] 
to be the the \emph{stretch factor} of a shortest immersed loop that represents $\alpha$.  
Define the \emph{tension subgraph}, $x_{\phi}$, or $\green(x, y)$, to be the 
subgraph of $x$ consisting of maximally stretched edges. 

Now we restate and prove part (i) of \thmref{Intro:CV}:

\begin{theorem}[Shadow of a geodesic in $\CV$ is not a quasi-geodesic in 
$\Out$] \label{Thm:CV-bad}
For given constants $K$ and $C$, there are points $x, y \in \CV$ such that for every 
geodesic $[x, y]_{\CV}$ in $\CV$ connecting $x$ to $y$, its image in $\Out$ is not 
$(K, C)-$quasi-geodesic. 
\end{theorem} 

\begin{proof}
Consider the same example in rank 3 as in \thmref{Thm:Stallings-bad} where $\phi$ is 
defined as

\[
 \left \{
 \begin{array}{ll}a \longrightarrow a  \\
  b \longrightarrow b\\
  c \longrightarrow c(\psi^q(a))^t
\end{array}
\right.
\]

Let $w=c\, (\psi^q(a))^t\in \mathbb{F}_3$ and let $M$ be the length of 
$w$ in the basis $\langle a, b, c \rangle$. Let $x \in \CV$ be a rose where the edges 
are labeled $a$, $b$ and $w$ and
\[
\ell_x(a) = \ell_x(b) =  \frac{1}{M+2}
\qquad\text{and}\qquad 
 \ell_x(w)= \frac{M}{M+2}
\]
and let $y \in \CV$ be a rose where the edges are labeled $a$, $b$ and $c$ 
and have length $\frac 13$ each. Note that the length ratio of $a$, $b$ and
$w$ from $x$ to $y$ are identical, 
\[
\frac{\ell_y(w)}{\ell_x(w)}= 
\frac{M/3}{M/(M+2)}=\frac{M +2}{3} = \frac{1/3}{1/(M+2)}
= \frac{\ell_y(a)}{\ell_x(a)}= \frac{\ell_y(b)}{\ell_x(b)}.
\]
In particular, we have
\[
d_{\CV(x, y)} = \log \frac{M +2}{3}. 
\]
In fact, it follows from \cite{convexball} that there is a unique geodesic $[x,y]_{\CV}$ in $\CV$ connecting 
$x$ to $y$ and it folds along the unique illegal turn. Namely, it folds the edge labeled 
$w$ around the free factor $\calA$ and if $p \from [0,m] \to \Out$ is the shadow of 
$[x,y]_{\CV}$ in $\Out$, then the projection of $p$ to $\calA=\langle a, b \rangle$ backtracks. 
Hence, as was seen in the proof of \thmref{Thm:Stallings-bad}, for any $K$ and $C$, 
we can choose $q$ and $t$ large enough so that $p$ is not a $(K,C)$--quasi-geodesics. 
\end{proof}

We can also modify the example in \thmref{Thm:Stallings-bad} to prove part (ii) 
of Theorem D. For brevity, we do not define greedy folding paths here. 
They are used in \cite{BFhyperbolicity} in an essential way to prove the hyperbolicity of 
of the free-factor graph. What we need is that if $\green(x,y) = x$ then there is a
greedy folding path connecting $x$ to $y$. 

\begin{theorem} 
There are points $x$ and $y$ in the thick part of $\CV$,  such that they are connected by a greedy
folding path whose shadow in $\Out$ is not a quasi-geodesic. 
\end{theorem} 

\begin{proof} 
Let $ \langle a, b, c \rangle $ be a generating set for $\mathbb{F}_3$,  let $R_0$ be 
a rose with labels $\{a, b, c \}$. Let $\psi$ and $\psi_\calA$ be as before. 
Let $y\in \CV$ be the rose $R_0$ with edge-labels $a$, $b$ and $c$ and edge-lengths 
$\frac 13$. 

Let $\Gg$ be the axis of $\psi_\calA$ in $\calF(\calA)$, 
the free factor graph associated to $\calA$. Also let $\alpha_1, \alpha_2, \dots, \alpha_k \in \calA$
be primitive loops in $\calA$ that, considered as vertices in $\calF(\calA)$, are distance
$D$ or further from $\Gg$ for a large constant $D$. 
\footnote{As we shall see, it is enough
to have one such loop, however choosing many loops showcases how different
the shadow of a quasi-geodesic in $\Out$ to $\calF(\calA)$ could be from being
a quasi-geodesic.} For positive integers $n_1, n_2, \dots, n_k$, let $\phi$ be the 
following automorphism:
\[
 \left \{
 \begin{array}{ll}a \longrightarrow \psi^n(a)  \\
  b \longrightarrow \psi^n(b)\\
  c \longrightarrow c \, \alpha_1^{n_1} \dots \alpha_k^{n_k}
\end{array}
\right.
\]

We observe that $| \psi^n(a)|_y$ and $|\psi^n(b)|_y$, as a function 
$n$, grow at a fixed exponential rate that is less than $3$. Therefore for any given loops 
$\alpha_i$ and powers $n_i$, there is a power $n$ so that
\begin{equation} \label{Eq:Ratios}
\max\big(| \psi^n(a)|_y, |\psi^n(b)|_y \big) \leq
| c \, \alpha_1^{n_1} \dots \alpha_k^{n_k}|_y  \leq 
3 \, \min\big(| \psi^n(a)|_y, |\psi^n(b)|_y \big), 
\end{equation}
Let $x$ be a rose with 
edge labels, $\psi^n(a)$, $\psi^n(b)$ and $c \, \alpha_1^{n_1} \dots \alpha_k^{n_k}$ 
(note that these form a basis for $\FF_3$) and edge lengths
\[ 
\frac {| \psi^n(a)|_y}{T}, \qquad \frac {|\psi^n(b)|_y}{T} \qquad\text{and}\qquad
\frac {| c \, \alpha_1^{n_1} \dots \alpha_k^{n_k}|_y}{T}. 
\]
where \[T=| \psi^n(a)|_y+ |\psi^n(b)|_y + | c \, \alpha_1^{n_1} \dots \alpha_k^{n_k}|_y.\] 
Then $\green(x,y)=x$ and \eqnref{Eq:Ratios} implies that every edge length in 
$x$ is larger than $\frac 1{7}$ ($x$ is $\frac 1{7}$--thick part of $\CV$). 

Let $\phi_{i}$ be the shortest automorphism such that $\phi_i (\alpha_i) = a$ and let $r_{i} = \Vert \phi_{i} \Vert$. It follows from construction that, 
\begin{equation} \label{Eq:Length}
\Vert \phi \Vert \leq 3n + \sum_{i=1}^k (2r_i + n_i).  
\end{equation}
Let $\ell_i$ be the length of $\alpha_i$ in the $\langle a, b \rangle$ basis and let
$\lambda$ be the stretch factor of $\phi$. To make \eqnref{Eq:Ratios} hold, 
we need 
\[
\lambda^n \sim 1+ \sum_{i=1}^k n_i \ell_i. 
\]
Letting $n_{\rm max} = \max_i n_i$ we have, for constants
$c_1$ and $c_2$ depending on $k$, $r_i$ and $\ell_i$, that 
\[
n \leq c_1 \log n_{\rm max}, \qquad\text{and}\qquad 
\Vert \phi \Vert  \leq (k+1) n_{\rm max}+c_2. 
\]

Now let $p\from [0,m] \to \Out$ be the shadow of $[x,y]_{\rm gf}$ to 
$\Out$. Bestvina-Feighn\cite{BFhyperbolicity} showed that the projection of $[x,y]_{\rm gf}$ to $\calF(\calA)$
is a quasi-geodesic. Hence the projection of $p[0,m]$ to $\calF(\calA)$ 
stays near $\Gg$ and thus remains far from every $\alpha_i$. 
And, again by  Lemma 3.3 in \cite{BFhyperbolicity}, this implies that the combinatorial
length of $\alpha_i$ at any point along $[x,y]_{\rm gf}$ is large, say larger
than some constant $L$ depending linearly on $D$. By \thmref{slowchange}, 
\[
m \succ \max_i n_i \log L. 
\]
But, since $p$ is a quasi-geodesic, $m \prec \Vert \phi \Vert$. For large enough $L$, 
the above inequality contradicts \eqnref{Eq:Length}. This implies $p$ cannot be a quasi-geodesic. 
\end{proof}

\bibliographystyle{alpha}

\end{document}